\documentclass{amsart}

\usepackage{amsmath,amssymb,yhmath, url, microtype,enumitem}

\makeatletter

\let\ssection\subsection
\renewcommand\subsection{\@startsection{subsection}{2}%
  \z@{.5\linespacing\@plus\linespacing}{.5\linespacing}%
  {\normalfont\bfseries}}
\makeatother

\newtheorem{lemma}{Lemma}
\newtheorem{prop}[lemma]{Proposition}

\newtheorem{thm}[lemma]{Theorem}

\newtheorem{thm?}[lemma]{Theorem?}

\newtheorem{remark}[lemma]{Remark}

\begin{document}
\title{The Truth About Torsion In The CM Case}
\author{Pete L. Clark}
\author{Paul Pollack}


%

\newcommand{\etalchar}[1]{$^{#1}$}
\newcommand{\F}{\mathbb{F}}
\newcommand{\et}{\textrm{\'et}}
\newcommand{\ra}{\ensuremath{\rightarrow}}
\newcommand{\FF}{\F}
\newcommand{\ff}{\mathfrak{f}}
\newcommand{\Z}{\mathbb{Z}}
\newcommand{\N}{\mathcal{N}}
\newcommand{\ch}{}
\newcommand{\R}{\mathbb{R}}
\newcommand{\PP}{\mathbb{P}}
\newcommand{\pp}{\mathfrak{p}}
\newcommand{\C}{\mathbb{C}}
\newcommand{\Q}{\mathbb{Q}}
\newcommand{\ab}{\operatorname{ab}}
\newcommand{\Aut}{\operatorname{Aut}}
\newcommand{\gk}{\mathfrak{g}_K}
\newcommand{\gq}{\mathfrak{g}_{\Q}}
\newcommand{\OQ}{\overline{\Q}}
\newcommand{\Out}{\operatorname{Out}}
\newcommand{\End}{\operatorname{End}}
\newcommand{\Gal}{\operatorname{Gal}}
\newcommand{\CT}{(\mathcal{C},\mathcal{T})}
\newcommand{\lcm}{\operatorname{lcm}}
\newcommand{\Div}{\operatorname{Div}}
\newcommand{\OO}{\mathcal{O}}
\newcommand{\rank}{\operatorname{rank}}
\newcommand{\tors}{\operatorname{tors}}
\newcommand{\IM}{\operatorname{IM}}
\newcommand{\CM}{\mathbf{CM}}
\newcommand{\HS}{\mathbf{HS}}
\newcommand{\Frac}{\operatorname{Frac}}
\newcommand{\Pic}{\operatorname{Pic}}
\newcommand{\coker}{\operatorname{coker}}
\newcommand{\Cl}{\operatorname{Cl}}
\newcommand{\loc}{\operatorname{loc}}
\newcommand{\GL}{\operatorname{GL}}
\newcommand{\PSL}{\operatorname{PSL}}
\newcommand{\Frob}{\operatorname{Frob}}
\newcommand{\Hom}{\operatorname{Hom}}
\newcommand{\Coker}{\operatorname{\coker}}
\newcommand{\Ker}{\ker}
\newcommand{\g}{\mathfrak{g}}
\newcommand{\sep}{\operatorname{sep}}
\newcommand{\new}{\operatorname{new}}
\newcommand{\Ok}{\mathcal{O}_K}
\newcommand{\ord}{\operatorname{ord}}
\newcommand{\mm}{\mathfrak{m}}
\newcommand{\Ohell}{\OO_{\ell^{\infty}}}
\newcommand{\cc}{\mathfrak{c}}
\newcommand{\ann}{\operatorname{ann}}
\renewcommand{\tt}{\mathfrak{t}}
\renewcommand{\cc}{\mathfrak{a}}
\renewcommand{\aa}{\mathfrak{a}}
\newcommand\leg{\genfrac(){.4pt}{}}
\newcommand{\TCMd}{T_{\CM}(d)}

\begin{abstract}
Let $\TCMd$ be the maximum size of the torsion subgroup of an elliptic
curve with complex multiplication defined over a degree $d$ number field.  We show that there is
an absolute, effective constant $C$ such that $\TCMd \leq C d \log \log d$
for all $d \geq 3$.
\end{abstract}

\maketitle
\noindent
\noindent
For a commutative group $G$, we denote by $G[\tors]$ the torsion subgroup of $G$.


\section{Introduction}
\noindent
The aim of this note is to prove the following result.

\begin{thm}
\label{MAINTHM}
There is an absolute, effective constant $C$ such that for all number fields $F$ of degree $d \geq 3$ and all elliptic curves $E_{/F}$ with complex multiplication,
\[ \# E(F)[\tors] \leq C d \log \log d. \]
\end{thm}
\noindent
It is natural to compare this result with the following one.

\begin{thm}[Hindry--Silverman \cite{HS99}]
\label{HSTHM}
For all number fields $F$ of degree $d \geq 2$ and all elliptic curves $E_{/F}$ with $j$-invariant $j(E) \in \OO_F$, we have
%
\[
 \# E(F)[\tors] \leq  1977408 d \log d.
\]
\end{thm}
\noindent
Every CM elliptic curve $E_{/F}$ has $j(E) \in \OO_F$, but only finitely many $j \in \OO_F$ are $j$-invariants of CM elliptic
curves $E_{/F}$.  Thus Theorem \ref{MAINTHM} has a significantly stronger
hypothesis and a slightly stronger conclusion than Theorem \ref{HSTHM}.  But
the improvement of $\log \log d$ over $\log d$ is interesting in view of
the following result.

\begin{thm}[Breuer \cite{Breuer10}] \label{BREUERTHM}
Let $E_{/F}$ be an elliptic curve over a number field.  There
exists a constant $c(E,F) > 0$, integers $3 \leq d_1 < d_2 < \ldots < d_n < \ldots$ and number fields $F_n \supset F$
with $[F_n:F] = d_n$ such that for all $n \in \Z^+$ we have
\[  \# E(F_n)[\tors] \geq
\begin{cases}
     c(E,F) d_n \log \log d_n \ & \text{ if $E$ has CM}, \\
     c(E,F) \sqrt{d_n \log \log d_n}   & \text{ otherwise}.
   \end{cases} \]

\end{thm}
\noindent
Let $T_{\CM}(d)$ be the maximum size of the torsion subgroup of a CM
elliptic curve over a degree $d$ number field.   Then
Theorems \ref{MAINTHM} and \ref{BREUERTHM} combine to tell us that the \emph{upper order} of $T_{\CM}(d)$ is $d \log \log d$:
\[
 0 < \limsup_{d \ra \infty} \frac{ T_{\CM}(d)}{d \log \log d} < \infty.
\]
\noindent
To our knowledge, this is the first instance of an upper order result for torsion points on a class of abelian varieties over number fields of varying degree.
\\ \\
Define $T(d)$ as for $T_{\CM}(d)$ but replacing ``CM elliptic curve''
with ``elliptic curve'', and define $T_{\neg \CM}(d)$ as for $T_{\CM}(d)$
but replacing ``CM elliptic curve'' with ``elliptic curve \emph{without} CM''.
Hindry and Silverman ask whether $T_{\neg \CM}(d)$ has upper order $\sqrt{d \log \log d}$.  If so, the upper order of $T(d)$ would be $d \log \log d$ \cite[Conjecture 1]{TORS1}.

\section{Proof of the Main Theorem}

\subsection{Torsion Points and Ray Class Containment}\label{sec:torsionandrayclass}


\noindent Let $K$ be a number field.  Let $\OO_K$ be the ring of integers of $K$, $\Delta_K$ the discriminant of $K$, $w_K$ the number of roots of unity in $K$ and $h_K$
the class number of $K$.  By an ``ideal of $\OO_K$'' we shall always mean a nonzero
ideal.  For an ideal $\cc$ of $\OO_K$, we write $K^{(\cc)}$ for the $\cc$-ray class field
of $K$.  We also put $|\cc| = \# \OO_K/\cc$ and

\[ \varphi_K(\cc) = \# (\OO_K/\cc)^{\times} = |\cc| \prod_{\pp \mid \cc} \left(1-\frac{1}{|\pp|} \right). \]
An elliptic curve $E$ defined over a field of characteristic $0$ has \emph{complex multiplication (CM)} if $\End E \supsetneq \Z$; then $\End E$ is an order in an imaginary quadratic field.  We say $E$ has $\OO$-CM
if $\End E \cong \OO$ and $K$-CM if $\End E$ is an order in $K$.


\begin{lemma}
\label{COMPOSITELEMMA}
Let $K$ be an imaginary quadratic field and $\aa$ an ideal of $\OO_K$. Then
\[ \frac{h_K \varphi_K(\cc)}{6} \leq \frac{h_K\varphi_K(\cc)}{w_K} \leq [K^{(\cc)}:K] \leq h_K \varphi_K(\cc). \]
\end{lemma}
\begin{proof}
This follows from \cite[Corollary 3.2.4]{Cohen2}.
\end{proof}

\begin{thm}
\label{RAYCLASSTHM}
Let $K$ be an imaginary quadratic field, $F \supset K$ a number field, $E_{/F}$ a $K$-CM elliptic curve and $N \in \Z^+$.  If $(\Z/N\Z)^2 \hookrightarrow E(F)$,
then $F \supset K^{(N \OO_K)}$.
\end{thm}


\begin{proof}
The result is part of classical CM theory when $\End E = \OO_K$ is the maximal order
in $K$ \cite[II.5.6]{SilvermanII}.  We shall  reduce to that case.  There is an $\OO_K$-CM elliptic curve $E'_{/F}$ and a canonical $F$-rational isogeny
$\iota\colon E \ra E'$ \cite[Prop. 25]{TORS1}.  There is a field embedding $F \hookrightarrow \C$ such that the base change of $\iota$ to $\C$ is, up to isomorphisms on the source and target, given by $\C/\OO \ra \C/\OO_K$.  If we put
\[  P = 1/N + \OO \in E[N], \quad P' = 1/N + \OO_K \in E'[N], \]
then $\iota(P) = P'$ and $P'$ generates $E'[N]$ as an $\OO_K$-module.  By assumption $P \in E(F)$, so $\iota(P) = P' \in E'(F)$.  It follows
that $(\Z/N\Z)^2 \hookrightarrow E'(F)[\tors]$.
\end{proof}

\begin{remark}
In fact one can show --- e.g., using adelic methods --- that for any $K$-CM elliptic curve $E$ defined over $\C$, the field obtained by adjoining to
$K(j(E))$ the values of the Weber function at the $N$-torsion points of $E$
contains $K^{(N \OO_K)}$.
\end{remark}




%


\subsection{Squaring the Torsion Subgroup of a CM Elliptic Curve}



\begin{thm}
\label{FIXITTHM}
Let $K$ be an imaginary quadratic field, let $F \supset K$ a field extension,
and let $E_{/F}$ be a $K$-CM elliptic curve.   Suppose that for positive integers $a$ and $b$ we have an injection
$\Z/a\Z \times \Z/ab\Z \hookrightarrow E(F)$.  Then $[F(E[ab]):F] \leq b$.
\end{thm}
\begin{proof}\textbf{Step 1}: Let $\OO = \End E$.  For $N \in \Z^+$, let $C_N = (\OO/N\OO)^{\times}$.  Let $E[N] = E[N](\overline{F})$.  As an $\OO/N\OO$-module, $E[N]$ is free of rank $1$.  Let $\mathfrak{g}_F = \Aut(\overline{F}/F)$, and let $\rho_N\colon \mathfrak{g}_F \longrightarrow \GL_2(\Z/N\Z)$ be the mod $N$ Galois representation associated to $E_{/F}$.
Because $E$ has $\OO$-CM and $F \supset K$, we have
\[
  \rho_N\colon \mathfrak{g}_F \longrightarrow \Aut_{\mathcal O} E[N] \cong \GL_1(\mathcal O/N\mathcal O)
         \cong (\mathcal O/N\mathcal O)^\times = C_N.
\]
Let $\Delta$ be the discriminant of $\OO$.  Then $e_1 = 1$, $e_2 = \frac{\Delta + \sqrt{\Delta}}{2}$ is a $\Z$-basis for $\OO$.  The induced ring embedding $\OO \hookrightarrow M_2(\Z)$ is given by $\alpha e_1 + \beta e_2 \mapsto \left[ \begin{smallmatrix} \alpha & \frac{\beta \Delta-\beta \Delta^2}{4} \\
\beta & \alpha + \beta \Delta \end{smallmatrix} \right]$.
So
\begin{multline*} C_N = \bigg\{ \left[ \begin{matrix} \alpha & \frac{\beta \Delta-\beta \Delta^2}{4} \\
\beta & \alpha + \beta \Delta \end{matrix} \right] \mid \alpha,\beta \in \Z/N\Z, \text{ and } \\
\alpha^2 + \Delta \alpha \beta + \left(\frac{\Delta^2-\Delta}{4}\right)\beta^2 \in (\Z/N\Z)^{\times}  \bigg\}. \end{multline*}
From this we easily deduce the following useful facts:
\begin{enumerate}
\item[(i)] $C_N$ contains the homotheties $\left\{ \left[ \begin{smallmatrix} \alpha & 0 \\ 0 & \alpha \end{smallmatrix} \right] \mid \alpha \in (\Z/N\Z)^{\times} \right\}$.
\item[(ii)] For all primes $p$ and all $A,B \geq 1$, the natural reduction map
$C_{p^{A+B}} \ra C_{p^A}$ is surjective and its kernel has size $p^{2B}$.
\end{enumerate}
\noindent \textbf{Step 2}: Primary decomposition reduces us to the
case $a = p^A$, $b = p^B$ with $A \geq 0$ and $B \geq 1$.  By induction it suffices to treat the case $B =1$: i.e., we assume $E(F)$ contains
full $p^A$-torsion and a point of order $p^{A+1}$ and show  $[F(E[p^{A+1}]):F] \leq p$.

 \vskip 2pt
\noindent\textbf{Case $A = 0$}:
\begin{itemize}[leftmargin=2em]
\item If $\left( \frac{\Delta}{p} \right) = 1$, then $C_p$ is conjugate
to $\left\{  \left[ \begin{smallmatrix} \alpha & 0 \\ 0 & \beta \end{smallmatrix} \right] \mid \alpha,\beta \in \F_p^{\times} \right\}$.  If $\alpha \neq 1$ (resp. $\beta \neq 1$) the only fixed points $(x,y) \in \F_p^2$ of $\left[ \begin{smallmatrix} \alpha & 0 \\ 0 & \beta \end{smallmatrix} \right] $ have $x = 0$
(resp. $y = 0$).  Because $E(F)$ contains a point of order
$p$ we must either have $\alpha = 1$ for all $\left[ \begin{smallmatrix} \alpha & 0 \\ 0 & \beta \end{smallmatrix} \right] \in \rho_p(\mathfrak{g}_F)$
or $\beta = 1$ for all $\left[ \begin{smallmatrix} \alpha & 0 \\ 0 & \beta \end{smallmatrix} \right] \in \rho_p(\mathfrak{g}_F)$.  Either way,
$\# \rho_p(\mathfrak{g}_F) \mid p-1$.
\item If $\left( \frac{\Delta}{p} \right) = -1$, then $C_p$ acts simply
transitively on $E[p] \setminus \{0\}$, so if we have one $F$-rational
point of order $p$ then $E[p] \subset E(F)$, so $\# \rho_p(\mathfrak{g}_F) = 1$.
\item If $\left( \frac{\Delta}{p} \right) = 0$, then $C_p$ is conjugate to $\left\{  \left[ \begin{smallmatrix} \alpha & \beta \\ 0 & \alpha \end{smallmatrix} \right] \mid \alpha \in \F_p^{\times}, \ \beta \in \F_p \right\}$ \cite[$\S 4.2$]{BCS15}.  Since $E(F)$ has a point of order $p$, every element of
$\rho_p(\mathfrak{g}_F)$ has $1$ as an eigenvalue and thus $\rho_p(\mathfrak{g}_F)
\subset \{  \left[ \begin{smallmatrix} 1 & \beta \\ 0 & 1 \end{smallmatrix} \right] \mid \beta \in \F_p \}$, so has order dividing $p$. \end{itemize}

\noindent\textbf{Case $A \geq 1$}: By (ii), $\mathcal{K} = \Ker C_{p^{A+1}} \ra C_{p^{A}}$ has size $p^2$.  Since $(\Z/p^{A}\Z)^2 \hookrightarrow E(F)$, we have $\rho_{p^{A+1}}(\mathfrak{g}_F) \subset \mathcal{K}$.  Since $E(F)$ has a point of order $p^{A+1}$, by (i) the homothety
$ \left[ \begin{smallmatrix} 1+p^{A} & 0 \\ 0 & 1 + p^{A} \end{smallmatrix} \right]$ lies in $\mathcal{K} \setminus \rho_{p^{A+1}}(\mathfrak{g}_F)$.
Therefore $\rho_{p^{A+1}}(\mathfrak{g}_F) \subsetneq
\mathcal{K}$, so $\# \rho_{p^{A+1}}(\mathfrak{g}_F) \mid p$.
\end{proof}

\subsection{Uniform Bound for Euler's Function in Imaginary Quadratic Fields}

\noindent Let $\aa$ be an ideal in an imaginary quadratic field $K$. To apply the results of \S\ref{sec:torsionandrayclass}, we require a lower bound on $\frac{\varphi_K(\aa)}{|\aa|}$. For \emph{fixed} $K$, it is straightforward to adapt a classical argument of Landau (see the proof of \cite[Theorem 328, p. 352]{HW}).  Replacing Landau's use of Mertens' Theorem with Rosen's number field analogue \cite{Rosen99}, one obtains the following result:
let $\gamma$ denote the Euler--Mascheroni constant, and let $\chi(\cdot) = \leg{\Delta_K}{\cdot}$ be the quadratic Dirichlet character associated to $K$. Then
\[ \liminf_{|\aa|\to\infty} \frac{ \varphi_K(\aa)}{|\aa|/\log \log |\aa|} = e^{-\gamma} \cdot L(1,\chi)^{-1}. \]
\noindent Unfortunately, this result is not sufficient for our purposes. There are two sources of difficulty. First, the right-hand side depends on $K$, and can in fact be arbitrarily small (see \cite[(4$'$)]{BCE50}).
Second, the statement only addresses limiting behavior as $|\aa|\to\infty$, and we need a result with no such restriction on $|\aa|$. However, looking back at Lemma \ref{COMPOSITELEMMA} we see that a lower bound on $h_K \frac{\varphi_K(\aa)}{|\aa|}$ would suffice.   The factor of
$h_K$ allows us to prove a totally uniform lower bound.

\begin{thm}\label{thm:phibound}
\label{PARTTHREE}
There is a positive, effective absolute constant $C$ such that for all imaginary quadratic fields $K$ and all nonzero
ideals $\cc$ of $\OO_K$ with $|\cc| \geq 3$, we have
\[ \varphi_K(\cc) \geq \frac{C}{h_K}\cdot \frac{|\cc|}{\log\log|\cc|}. \]
\end{thm}
\noindent

\begin{lemma}\label{lemma:lowerprod} For a fundamental quadratic discriminant $\Delta < 0$  let $K=\Q(\sqrt{\Delta})$, and let $\chi(\cdot) = \leg{\Delta}{\cdot}$.  There is an effective constant $C > 0$ such that for all $x \ge 2$,
\begin{equation}\label{eq:lowerprod} \prod_{p \le x}\left(1-\frac{\chi(p)}{p}\right) \geq \frac{C}{h_K}. \end{equation}
\end{lemma}
\begin{proof} By the quadratic class number formula, $h_K \asymp L(1,\chi)\sqrt{|\Delta|}$ \cite[eq. (15), p. 49]{Davenport}. Writing $L(1,\chi) = \prod_{p}(1-\chi(p)/p)^{-1}$ and rearranging, we see \eqref{eq:lowerprod} holds iff
\begin{equation}
\label{PAULEQ}
 \prod_{p > x} \left(1-\frac{\chi(p)}{p}\right) \ll \sqrt{|\Delta|},
\end{equation}
with an effective and absolute implied constant. By Mertens' Theorem \cite[Theorem 429, p. 466]{HW}, the factors on the left-hand side of (\ref{PAULEQ}) indexed by $p \le \exp(\sqrt{|\Delta|})$  make a contribution of $O(\sqrt{|\Delta|})$.  Put $y= \max\{x, \exp(\sqrt{|\Delta|})\}$; it suffices to show that $\prod_{p > y} \left(1-\chi(p)/p\right) \ll 1$. Taking logarithms, this will follow if we prove that $\sum_{p > y} \chi(p)/p=O(1)$. For $t \ge \exp(\sqrt{|\Delta|})$, the explicit formula gives $S(t):= \sum_{p \le t} \chi(p) \log{p} = -t^{\beta}/\beta + O(t/\log{t})$, where the main term is present only if $L(s,\chi)$ has a Siegel zero $\beta$. (C.f. \cite[eq. (8), p. 123]{Davenport}.) We will assume the Siegel zero exists; otherwise the argument is similar but simpler. By partial summation,
\begin{align*} \sum_{p>y} \frac{\chi(p)}{p} &= -\frac{S(y)}{y\log{y}} + \int_{y}^{\infty} \frac{S(t)}{t^2 (\log{t})^2} (1+\log{t})\, dt \\&\ll 1 + \int_{y}^{\infty} \frac{t^{\beta}}{t^2 \log{t}}\, dt.
\end{align*}
Haneke, Goldfeld--Schinzel, and Pintz have each shown that $\beta \le 1-c/\sqrt{|\Delta|}$, where the constant $c> 0$ is absolute and effective \cite{Haneke73,GS75,Pintz76}. Using this to bound $t^{\beta}$, and keeping in mind that $y \ge \exp(\sqrt{|\Delta|})$, we see that the final integral is at most
\[ \int_{\exp(\sqrt{|\Delta|})}^{\infty} \frac{\exp(-c\log{t}/\sqrt{|\Delta|})}{t\log{t}} \,dt. \]
A change of variables transforms the integral into $\int_{1}^{\infty} \exp(-cu) u^{-1}\,du$, which converges. Assembling our estimates completes the proof.
\end{proof}

\begin{proof}[Proof of Theorem \ref{thm:phibound}] Write $\varphi_K(\cc) = |\cc| \prod_{\pp \mid \cc} (1-1/|\pp|)$, and notice that the factors are increasing in $|\pp|$. So if $z\ge 2$ is such that $\prod_{|\pp| \le z} |\pp| \ge |\cc|$, then
\begin{equation}\label{eq:step0} \frac{\varphi_K(\cc)}{|\cc|} \ge \prod_{|\pp| \le z} \bigg(1-\frac{1}{|\pp|}\bigg). \end{equation}
We first establish a lower bound on the right-hand side, as a function of $z$, and then we prove the theorem by making a convenient choice of $z$.  We partition the prime ideals with $|\pp|\le z$ according to the splitting behavior of the rational prime $p$ lying below $\pp$. Noting that $p \le |\pp|$, Mertens' Theorem and Lemma \ref{lemma:lowerprod} yield
\begin{align} \prod_{|\pp| \le z} \bigg(1-\frac{1}{|\pp|}\bigg) &\ge \prod_{p \le z} \bigg(1-\frac{1}{p}\bigg) \bigg(1-\frac{\leg{\Delta}{p}}{p}\bigg) \notag\\
&\gg (\log{z})^{-1} \prod_{p \le z} \bigg(1-\frac{\leg{\Delta}{p}}{p}\bigg) \gg (\log{z})^{-1} \cdot h_K^{-1}. \label{eq:step1}\end{align}
With $C'$ a large absolute constant to be described momentarily, we set
\begin{equation}\label{eq:step2} z=(C'\log |\cc|)^2.\end{equation}
We must check that $\prod_{|\pp| \le z}|\pp| \ge |\cc|$.  The Prime Number Theorem implies \[ \prod_{|\pp| \le z}|\pp| \ge \prod_{p \le z^{1/2}} p \ge \prod_{p \le C'\log |\cc|}p \ge |\cc|,\] provided that $C'$ was chosen appropriately.  Combining \eqref{eq:step0}, \eqref{eq:step1}, and \eqref{eq:step2} gives \[\varphi_K(\cc)\gg |\cc| \cdot (\log{z})^{-1} \cdot h_K^{-1}\gg h_K^{-1} \cdot |\cc| \cdot \log(\log(|\cc|))^{-1}. \qedhere \]
\end{proof}











\subsection{Proof of Theorem \ref{MAINTHM}}
\noindent Let $F$ be a number field of degree $d \geq 3$, and let $E_{/F}$ be
a $K$-CM elliptic curve.  We may assume $\# E(F)[\tors] \geq 3$.  We have $E(FK)[\tors] \cong \Z/a\Z \times \Z/ab\Z$ for positive integers $a$ and $b$.   Theorem \ref{RAYCLASSTHM}
gives $FK \supset K^{(a \OO_K)}$.  Along with Lemma \ref{COMPOSITELEMMA} we get \[2d \geq [FK:\Q] \geq [K^{(a \OO_K)}:\Q] \geq \frac{h_K \varphi_K(a \OO_K)}{3}. \]
By
Theorem \ref{FIXITTHM}, there is an extension $L/FK$ with $(\Z/ab\Z)^2 \hookrightarrow E(L)$ and $[L:FK] \leq b$.  Applying Theorem \ref{RAYCLASSTHM} and Lemma \ref{COMPOSITELEMMA} as above we get
$L \supset K^{(ab \OO_K)}$ and
\[ [L:\Q] \geq [K^{(ab \OO_K)}:\Q] \geq  \frac{h_K \varphi_K(ab \OO_K)}{3}, \]
so
\begin{equation}
\label{FINALEQ0}
 d = [F:\Q]  \geq \frac{[FK:\Q]}{2} = \frac{[L:\Q]}{2[L:FK]} \geq \frac{[L:\Q]}{2b} \geq \frac{h_K \varphi_K(ab \OO_K)}{6b}.
\end{equation}
Multiplying (\ref{FINALEQ0}) through by $(ab)^2 = |ab \OO_K|$ and rearranging, we get
\begin{equation}
\label{FINALEQ1}
\# E(FK)[\tors] = a^2 b \leq  6\frac{ d}{h_K} \frac{|ab\OO_K|}{\varphi_K(ab\OO_K)}.
\end{equation}
By Theorem \ref{PARTTHREE} we have
\begin{equation}
\label{FINALEQ2}
 \frac{|ab \OO_K|}{\varphi_K(ab\OO_K)} \ll h_K \log \log |ab\OO_K|  \leq h_K \log \log (a^2 b)^2 \ll h_K \log \log \# E(FK)[\tors].
\end{equation}
Combining (\ref{FINALEQ1}) and (\ref{FINALEQ2}) gives
\[ \# E(FK)[\tors] \ll d \log \log \# E(FK)[\tors] \]
and thus
\[ \#E(F)[\tors] \leq \# E(FK)[\tors] \ll d \log \log d. \]

\ssection*{Acknowledgments} Thanks to John Voight for his encouragement and to Abbey Bourdon for pointing out an error in a previous draft. We are grateful to the referee for a careful reading of the manuscript.

\end{document}